\documentclass{article}

\usepackage{cmbright}
\usepackage{amssymb,amsmath,amsthm}
\usepackage{mathrsfs}
\usepackage{epsfig}
\usepackage[usenames,dvipsnames]{color}

\def\polhk#1{\setbox0=\hbox{#1}{\ooalign{\hidewidth\lower1.5ex\hbox{`}\hidewidth\crcr\unhbox0}}}

\def\B{\mathscr B}
\def\C{\mathbb C}
\def\d{\mathrm{d}}

\def\F{\mathscr F}

\def\H{\mathcal H}

\def\L{\mathscr L}

\def\N{\mathbb N}

\def\Q{\mathbb Q}
\def\R{\mathbb R}
\def\S{\mathbb S}
\def\SU{\mathsf{SU}}
\def\T{\mathbb T}
\def\U{\mathsf U}

\def\Z{\mathbb Z}

\def\dom{\mathcal D}

\def\e{\mathop{\mathrm{e}}\nolimits}

\def\linf{\mathsf{L}^{\:\!\!\infty}}
\def\ltwo{\mathsf{L}^{\:\!\!2}}

\def\slim{\mathop{\hbox{\rm s-}\lim}\nolimits}

\newtheorem{Theorem}{Theorem}[section]
\newtheorem{Remark}[Theorem]{Remark}
\newtheorem{Lemma}[Theorem]{Lemma}
\newtheorem{Assumption}[Theorem]{Assumption}

\begin{document}


\title{The absolute continuous spectrum of skew products of compact Lie groups}

\author{R. Tiedra de Aldecoa\footnote{Supported by the Chilean Fondecyt Grant
1130168, by the Iniciativa Cientifica Milenio ICM P07-027-F ``Mathematical Theory of
Quantum and Classical Magnetic Systems'' from the Chilean Ministry of Economy and by
the ECOS/CONICYT Grant C10E01.}}

\date{\small}
\maketitle \vspace{-1cm}

\begin{quote}
\emph{
\begin{itemize}
\item[] Facultad de Matem\'aticas, Pontificia Universidad Cat\'olica de Chile,\\
Av. Vicu\~na Mackenna 4860, Santiago, Chile
\item[] \emph{E-mail:} rtiedra@mat.puc.cl
\end{itemize}
}
\end{quote}


\begin{abstract}
Let $X$ and $G$ be compact Lie groups, $F_1:X\to X$ the time-one map of a $C^\infty$
measure-preserving flow, $\phi:X\to G$ a continuous function and $\pi$ a
finite-dimensional irreducible unitary representation of $G$. Then, we prove that the
skew products
$$
T_\phi:X\times G\to X\times G,\quad(x,g)\mapsto\big(F_1(x),g\;\!\phi(x)\big),
$$
have purely absolutely continuous spectrum in the subspace associated to $\pi$ if
$\pi\circ\phi$ has a Dini-continuous Lie derivative along the flow and if a matrix
multiplication operator related to the topological degree of $\pi\circ\phi$ has
nonzero determinant. This result provides a simple, but general, criterion for the
presence of an absolutely continuous component in the spectrum of skew products of
compact Lie groups. As an illustration, we consider the cases where $F_1$ is an
ergodic translation on $\T^d$ and $X\times G=\T^d\times\T^{d'}$,
$X\times G=\T^d\times\SU(2)$ and $X\times G=\T^d\times\U(2)$. Our proofs rely on
recent results on positive commutator methods for unitary operators.
\end{abstract}

\textbf{2010 Mathematics Subject Classification\;\!:} 37A30, 37C40, 58J51, 28D10.

\smallskip

\textbf{Keywords\;\!:} Skew products, compact Lie groups, continuous spectrum,
commutator methods.

\section{Introduction}
\setcounter{equation}{0}

In his seminal work \cite{Anz51}, H. Anzai has shown that the skew products
$$
T_m:\T\times\T\to\T\times\T,\quad(x,g)\mapsto(x+y,g+mx),
\quad\T\simeq\R/\Z,~y\in\R\setminus\Q,~m\in\Z\setminus\{0\},
$$
have countable Lebesgue spectrum in the orthocomplement of functions depending only
on the first variable. Since then, various generalisations of this result have been
obtained (see \cite{Cho87,Cho94,Fra95,Fra00,Fra00_2,Goo99,HP78,Iwa97_1,Iwa97_2,ILM99,
ILR93,Kus74,Med94} and see also \cite{AC00,Fay02,Iwa95,Wys04} for related results).
In particular, A. Iwanik, M. Lema\'nzyk and D. Rudolph have shown in \cite{ILR93}
that the skew products
$$
T_\phi:\T\times\T\to\T\times\T,\quad(x,g)\mapsto\big(x+y,g+\phi(x)\big),
\quad y\in\R\setminus\Q,
$$
have countable Lebesgue spectrum if the cocycle $\phi:\T\to\T$ is an absolutely
continuous function with nonzero topological degree and with derivative of bounded
variation. Also, A. Iwanik and K. Fr{\polhk{a}}czek have proved in
\cite{Fra95,Fra00,Iwa97_2} similar results for skew products of higher dimensional
tori. Finally, K. Fr{\polhk{a}}czek has shown in \cite{Fra00_2} (see also
\cite{Fra04}) that the skew products
$$
T_\phi:\T\times\SU(2)\to\T\times\SU(2),\quad(x,g)\mapsto\big(x+y,g\;\!\phi(x)\big),
\quad y\in\R\setminus\Q,
$$
have countable Lebesgue spectrum in an appropriate subspace of the orthocomplement
of functions depending only on the first variable if the cocycle $\phi:\T\to\SU(2)$
is of class $C^2$, has nonzero topological degree and is cohomologous (in a suitable
way) to a diagonal cocycle.

The purpose of this paper is to extend these types of spectral results to a general
class of skew products of compact Lie groups. Our set-up is the following. We
consider skew products
$$
T_\phi:X\times G\to X\times G,\quad(x,g)\mapsto\big(F_1(x),g\;\!\phi(x)\big),
$$
where $X$ and $G$ are compact Lie groups, $F_1:X\to X$ the time-one map of a
$C^\infty$ measure-preserving flow and $\phi:X\to G$ a continuous function. We fix
$\pi$ a finite-dimensional irreducible unitary representation of $G$ and we write
$\L_Y$ for the Lie derivative associated to the flow. Then, our main result reads as
follows. If the Lie derivative $\L_Y(\pi\circ\phi)$ exists and satisfies a Dini-type
condition along the flow and if a matrix multiplication operator related to the
topological degree of $\pi\circ\phi$ has nonzero determinant, then $T_\phi$ has
purely absolutely continuous spectrum in the subspace associated to $\pi$ (see
Theorem \ref{second_theorem} and Remark \ref{Rem_degree} for a precise satement).
This result provides a simple, but general, criterion for the presence of an
absolutely continuous component in the spectrum of skew products of compact Lie
groups. Its proof relies on recent results \cite{FRT12} on positive commutator
methods for unitary operators. As an illustration, we consider the cases where $F_1$
is an ergodic translation on $\T^d$ and $X\times G=\T^d\times\T^{d'}$,
$X\times G=\T^d\times\SU(2)$ and $X\times G=\T^d\times\U(2)$. In the cases
$X\times G=\T^d\times\T^{d'}$ and $X\times G=\T^d\times\SU(2)$, we obtain countable
Lebesgue spectrum under conditions similar to the ones to be found in the literature;
see Theorems \ref{thm_abelian} and \ref{thm_SU(2)} and the discussions that follow.
In the case $X\times G=\T^d\times\U(2)$, our result (countable Lebesgue spectrum in
an appropriate subspace) is new; see Theorem \ref{thm_U(2)} and the discussion that
follows.

Here is a brief description of the content of the paper. In Section
\ref{Sec_commutator}, we recall the needed notations and results on positive
commutator methods for unitary operators. In Section \ref{Sec_spectrum}, we construct
an appropriate conjugate operator (Lemma \ref{lemma_conjugate} and Formula
\eqref{second_operator}) and use it to prove our main theorem on the spectrum of skew
products (Theorem \ref{second_theorem}). We also give an interpretation of our result
in terms of the topological degree of $\pi\circ\phi$ (Remark \ref{Rem_degree}).
Finally, we present in Sections \ref{Sec_abelian}, \ref{Sec_SU(2)} and \ref{Sec_U(2)}
the examples $X\times G=\T^d\times\T^{d'}$, $X\times G=\T^d\times\SU(2)$ and
$X\times G=\T^d\times\U(2)$.

To conclude, we mention two possible extensions of this paper worth studying\;\!: (i)
We consider here skew products where the dynamics on the base space is given by the
time-one map of a flow. This guarantees the existence of a distinguished Lie
derivative which is used for the definition of the conjugate operator. It would be
interesting to see if one can still construct a suitable conjugate operator even if
the dynamics on the base space is not given by the time-one map of a flow. (ii) It
would be interesting to see if the result of this paper could be extended to the case
where $X$ and $G$ are noncompact Lie groups. In such a case, the main difficulty
would be to deal with the unavailability of Peter-Weyl theorem, which is repeatedly
used in this paper.\\

\noindent
{\bf Acknowledgements.} The author is grateful for the hospitality of Professors A.
Katok and F. Rodriguez Hertz at Penn State University in January 2013.

\section{Commutator methods for unitary operators}\label{Sec_commutator}
\setcounter{equation}{0}

We briefly recall in this section some facts on commutator methods for unitary
operators borrowed from \cite{FRT12}. We refer the reader to \cite{ABG,Mou81} for
standard references in the case of self-adjoint operators.

Let $\H$ be a Hilbert space with scalar product
$\langle\;\!\cdot\;\!,\;\!\cdot\;\!\rangle$ antilinear in the first argument, denote
by $\B(\H)$ the set of bounded linear operators on $\H$, and write
$\|\;\!\cdot\;\!\|$ both for the norm on $\H$ and the norm on $\B(\H)$. Let $A$ be a
self-adjoint operator in $\H$ with domain $\dom(A)$, and take $S\in\B(\H)$. For any
$k\in\N$, we say that $S$ belongs to $C^k(A)$, with notation $S\in C^k(A)$, if the
map
\begin{equation}\label{eq_group}
\R\ni t\mapsto\e^{-itA}S\e^{itA}\in\B(\H)
\end{equation}
is strongly of class $C^k$. In the case $k=1$, one has $S\in C^1(A)$ if and only if
the quadratic form
$$
\dom(A)\ni\varphi\mapsto\big\langle\varphi,iSA\;\!\varphi\big\rangle
-\big\langle A\;\!\varphi,iS\varphi\big\rangle\in\C
$$
is continuous for the topology induced by $\H$ on $\dom(A)$. We denote by $[iS,A]$
the bounded operator associated with the continuous extension of this form, or
equivalently the strong derivative of the function \eqref{eq_group} at $t=0$.

A condition slightly stronger than the inclusion $S\in C^1(A)$ is provided by the
following definition\;\!: $S$ belongs to $C^{1+0}(A)$, with notation
$S\in C^{1+0}(A)$, if $S\in C^1(A)$ and if $[A,S]$ satisfies the Dini-type condition
$$
\int_0^1\frac{\d t}t\;\!\big\|\e^{-itA}[A,S]\e^{itA}-[A,S]\big\|<\infty.
$$
As banachisable topological vector spaces, the sets $C^2(A)$, $C^{1+0}(A)$, $C^1(A)$
and $C^0(A)\equiv\B(\H)$ satisfy the continuous inclusions \cite[Sec.~5.2.4]{ABG}
$$
C^2(A)\subset C^{1+0}(A)\subset C^1(A)\subset C^0(A)\equiv\B(\H).
$$

Now, let $U\in C^1(A)$ be a unitary operator with (complex) spectral measure
$E^U(\;\!\cdot\;\!)$ and spectrum $\sigma(U)\subset\mathbb S^1:=\{z\in\C\mid|z|=1\}$.
If there exist a Borel set $\Theta\subset\S^1$, a number $a>0$ and a compact operator
$K\in\B(\H)$ such that
\begin{equation}\label{Mourre_U}
E^U(\Theta)\;\!U^*[A,U]E^U(\Theta)\ge a\;\!E^U(\Theta)+K,
\end{equation}
then one says that $U$ satisfies a Mourre estimate on $\Theta$ and that $A$ is a
conjugate operator for $U$ on $\Theta$. Also, one says that $U$ satisfies a strict
Mourre estimate on $\Theta$ if \eqref{Mourre_U} holds with $K=0$. The main
consequence of a strict Mourre estimate is to imply a limiting absorption principle
for (the Cayley transform of) $U$ on $\Theta$ if $U$ is also of class $C^{1+0}(A)$.
This in turns implies that $U$ has no singular spectrum in $\Theta$. If $U$ only
satisfies a Mourre estimate on $\Theta$, then the same holds up to the possible
presence of a finite number of eigenvalues in $\Theta$, each one of finite
multiplicity. We recall here a version of these results (see
\cite[Thm.~2.7 \& Rem.~2.8]{FRT12} for more details):

\begin{Theorem}[Mourre estimate for unitary operators]\label{thm_unitary}
Let $U$ and $A$ be respectively a unitary and a self-ajoint operator in a Hilbert
space $\H$, with $U\in C^{1+0}(A)$. Suppose there exist an open set
$\Theta\subset\mathbb S^1$, a number $a>0$ and a compact operator $K\in\B(\H)$ such
that
\begin{equation}\label{Mourre_U_bis}
E^U(\Theta)\;\!U^*[A,U]E^U(\Theta)\ge a\;\!E^U(\Theta)+K.
\end{equation}
Then, $U$ has at most finitely many eigenvalues in $\Theta$, each one of finite
multiplicity, and $U$ has no singular continuous spectrum in $\Theta$. Furthermore,
if \eqref{Mourre_U_bis} holds with $K=0$, then $U$ has no singular spectrum in
$\Theta$.
\end{Theorem}

\section{Spectrum of skew products of compact Lie groups}\label{Sec_spectrum}
\setcounter{equation}{0}

Let $X$ be a compact Lie group with normalised Haar measure $\mu_X$ and neutral
element $e_X$, and let $\{F_t\}_{t\in\R}$ be a $C^\infty$ measure-preserving flow on
$(X,\mu_X)$. Then, the family of operators $\{V_t\}_{t\in\R}$ given by
$$
V_t\;\!\varphi:=\varphi\circ F_t,\quad\varphi\in\ltwo(X,\mu_X),
$$
defines a strongly continuous one-parameter unitary group satisfying
$V_t\;\!C^\infty(X)\subset C^\infty(X)$ for each $t\in\R$ \cite[Prop.~2.6.14]{AM78}.
It follows from Nelson's theorem \cite[Prop.~5.3]{Amr09} that the generator $H$ of
the group $\{V_t\}_{t\in\R}$
$$
H\varphi:=\slim_{t\to0}it^{-1}(V_t-1)\varphi,
\quad\varphi\in\dom(H):=\left\{\varphi\in\ltwo(X,\mu_X)
\mid\lim_{t\to0}|t|^{-1}\big\|(V_t-1)\varphi\big\|<\infty\right\},
$$
is essentially self-adjoint on $C^\infty(X)$, and one has
$$
H\varphi:=-i\;\!\L_Y\varphi,\quad\varphi\in C^\infty(X),
$$
with $Y$ the divergence-free vector field associated to $\{F_t\}_{t\in\R}$ and $\L_Y$
the corresponding Lie derivative.

Let $G$ be a second compact Lie group with normalised Haar measure $\mu_G$ and
neutral element $e_G$. Then, each $\phi\in C(X;G)$ induces a cocycle
$X\times\Z\ni(x,n)\mapsto\phi^{(n)}(x)\in G$ over the diffeomorphism $F_1$ given by
$$
\phi^{(n)}(x):=
\begin{cases}
\phi(x)(\phi\circ F_1)(x)\cdots(\phi\circ F_{n-1})(x) & \hbox{if }n\ge1\\
e_G & \hbox{if }n=0\\
\big\{(\phi\circ F_n)(x)(\phi\circ F_{n+1})(x)\cdots(\phi\circ F_{-1})(x)\big\}^{-1}
& \hbox{if }n\le-1.
\end{cases}
$$
We thus call cocycle any function belonging to $C(X;G)$. Two cocycles
$\phi,\xi\in C(X;G)$ are $C^0$-cohomologous if there exists a function
$\zeta\in C(X;G)$, called transfer function, such that
$$
\phi(x)=\zeta(x)^{-1}\;\!\xi(x)\;\!\zeta\big(F_1(x)\big),\quad x\in X.
$$
In such a case, the map $X\times G\ni(x,g)\mapsto\big(x,\zeta(x)g\big)$ establishes a
$C^0$-conjugation of $T_\phi$ and $T_\xi$. The skew product $T_\phi$ associated to
$\phi$,
$$
T_\phi:(X\times G,\mu_X\otimes\mu_G)\to(X\times G,\mu_X\otimes\mu_G),\quad
(x,g)\mapsto\big(F_1(x),g\;\!\phi(x)\big),
$$
is an automorphism of the measure space $(X\times G,\mu_X\otimes\mu_G)$ which
satisfies
\begin{equation}\label{induc_T}
T_\phi^n(x,g)=\big(F_n(x),g\;\!\phi^{(n)}(x)\big),\quad x\in X,~g\in G,~n\in\Z.
\end{equation}
Since $T_\phi$ is invertible, the corresponding Koopman operator
$$
U_\phi\;\!\psi:=\psi\circ T_\phi,\quad\psi\in\H:=\ltwo(X\times G,\mu_X\otimes\mu_G),
$$
is a unitary operator in $\H$.

Let $\widehat G$ be the set of all (equivalence classes of) finite-dimensional
irreducible unitary representations (IUR) of $G$. Then, each representation
$\pi\in\widehat G$ is a $C^\infty$ group homomorphism from $G$ to the unitary group
$\U(d_\pi)$ of degree $d_\pi:=\dim(\pi)<\infty$, and the Peter-Weyl theorem implies
that the set of all matrix elements $\{\pi_{jk}\}_{j,k=1}^{d_\pi}$ of all
representations $\pi\in\widehat G$ form an orthogonal basis of $\ltwo(G,\mu_G)$.
Accordingly, one has the orthogonal decomposition
\begin{equation}\label{eq_decompo}
\H=\bigoplus_{\pi\in\widehat G}\bigoplus_{j=1}^{d_\pi}\H^{(\pi)}_j
\quad\hbox{with}\quad
\H^{(\pi)}_j:=\left\{\sum_{k=1}^{d_\pi}\varphi_k\otimes\pi_{jk}\mid
\varphi_k\in\ltwo(X,\mu_X),~k=1,\ldots,d_\pi\right\},
\end{equation}
and one has a natural isomorphism
\begin{equation}\label{eq_iso}
\H^{(\pi)}_j\simeq\bigoplus_{k=1}^{d_\pi}\ltwo(X,\mu_X),
\end{equation}
due to the orthogonality of the matrix elements $\{\pi_{jk}\}_{j,k=1}^{d_\pi}$.

A direct calculation shows that the operator $U_\phi$ is reduced by the decomposition
\eqref{eq_decompo} and that the restriction $U_{\pi,j}:=U_\phi|_{\H^{(\pi)}_j}$ is
given by
$$
U_{\pi,j}\sum_{k=1}^{d_\pi}\varphi_k\otimes\pi_{jk}
=\sum_{k,\ell=1}^{d_\pi}(V_1\varphi_k)(\pi_{\ell k}\circ\phi)\otimes\pi_{j\ell}\;\!,
\quad\varphi_k\in\ltwo(X,\mu_X).
$$
This, together with \eqref{induc_T}, implies that
$$
(U_{\pi,j})^n\sum_{k=1}^{d_\pi}\varphi_k\otimes\pi_{jk}
=\sum_{k,\ell=1}^{d_\pi}(V_n\;\!\varphi_k)\big(\pi_{\ell k}\circ\phi^{(n)}\big)
\otimes\pi_{j\ell},\quad n\in\Z,~\varphi_k\in\ltwo(X,\mu_X).
$$

\begin{Assumption}[Cocycle]\label{ass_phi}
For each $k,\ell\in\{1,\ldots,d_\pi\}$, the function
$\pi_{k\ell}\circ\phi\in C(X;\C)$ has a Lie derivative $\L_Y(\pi_{k\ell}\circ\phi)$
which satisfies the following Dini-type condition along the flow $\{F_t\}_{t\in\R}:$
$$
\int_0^1\frac{\d t}t\,\big\|\L_Y(\pi_{k\ell}\circ\phi)\circ F_t
-\L_Y(\pi_{k\ell}\circ\phi)\big\|_{\linf(X)}
<\infty.
$$
\end{Assumption}

Assumption \ref{ass_phi} implies that $-i\;\!\L_Y(\pi\circ\phi)\cdot(\pi^*\circ\phi)$
is a continuous hermitian matrix-valued function. In particular, the matrix
multiplication operator $M$ given by
$$
M\sum_{k=1}^{d_\pi}\varphi_k\otimes\pi_{jk}
:=-i\sum_{k,\ell=1}^{d_\pi}a_k\big\{\L_Y(\pi\circ\phi)
\cdot(\pi^*\circ\phi)\big\}_{k\ell}\;\!\varphi_\ell\otimes\pi_{jk},
\quad a_k\in\R,~\varphi_\ell\in\ltwo(X,\mu_X),
$$
is bounded in $\H^{(\pi)}_j$. We use the notation
\begin{equation}\label{formula_M}
M_{k\ell}:=-ia_k\big\{\L_Y(\pi\circ\phi)\cdot(\pi^*\circ\phi)\big\}_{k\ell}\;\!,
\quad k,\ell\in\{1,\ldots,d_\pi\},
\end{equation}
for the matrix elements of $M$.

\begin{Lemma}[Conjugate operator for $U_{\pi,j}$]\label{lemma_conjugate}
Let $\phi$ satisfy Assumption \ref{ass_phi} and suppose that
$(a_k-a_\ell)(\pi_{\ell k}\circ\phi)\equiv0$ for all $k,\ell\in\{1,\ldots,d_\pi\}$.
Then,
\begin{enumerate}
\item[(a)] The operator $A$ given by
$$
A\sum_{k=1}^{d_\pi}\varphi_k\otimes\pi_{jk}
:=\sum_{k=1}^{d_\pi}a_kH\varphi_k\otimes\pi_{jk},
\quad a_k\in\R,~\varphi_k\in C^\infty(X),
$$
is essentially self-adjoint in $\H^{(\pi)}_j$, and its closure (which we denote by
the same symbol) has domain
$$
\dom(A)=\left\{\sum_{k=1}^{d_\pi}\varphi_k\otimes\pi_{jk}\mid
\varphi_k\in\dom(H),~k=1,\ldots,d_\pi\right\}.
$$
Furthermore, one has
\begin{equation}\label{group_formula}
\e^{itA}\sum_{k=1}^{d_\pi}\varphi_k\otimes\pi_{jk}
=\sum_{k=1}^{d_\pi}\e^{ita_kH}\varphi_k\otimes\pi_{jk},
\quad t\in\R,~\varphi_k\in\ltwo(X,\mu_X).
\end{equation}
\item[(b)] For all $k,\ell\in\{1,\ldots,d_\pi\}$ and all $t\in\R$, one has
$$
\big\|\e^{-ita_kH}M_{k\ell}\e^{ita_\ell H}-M_{k\ell}\big\|_{\B(\ltwo(X,\mu_X))}
=\big\|\e^{-ita_kH}M_{k\ell}\e^{ita_kH}-M_{k\ell}\big\|_{\B(\ltwo(X,\mu_X))}.
$$
\item[(c)] $U_{\pi,j}\in C^{1+0}(A)$ with $[A,U_{\pi,j}]=MU_{\pi,j}$.
\end{enumerate}
\end{Lemma}

\begin{Remark}
(i) The commutation assumption $(a_k-a_\ell)(\pi_{\ell k}\circ\phi)\equiv0$ for all
$k,\ell\in\{1,\ldots,d_\pi\}$ implies that the matrix $M(x)$ is hermitian for each
$x\in X$, and thus that the operator $M$ is self-adjoint. This assumption is
satisfied if all the $a_k$'s are equal or if the matrix-valued function
$\pi\circ\phi$ is diagonal (which occurs for instance when $G$ is abelian). (ii)
Instead of the diagonal operator $A$, one could use the more general, non-diagonal,
self-adjoint operator
$$
A'\sum_{k=1}^{d_\pi}\varphi_k\otimes\pi_{jk}
:=\sum_{k=1}^{d_\pi}a_{k\ell}\;\!H\varphi_\ell\otimes\pi_{jk},
\quad\varphi_k\in C^\infty(X),~a_{k\ell}=\overline{a_{\ell k}}\in\C.
$$
However, doing this, one ends up with the same commutation relation in Lemma
\ref{lemma_conjugate}(c) (the scalars $a_k$ appearing in the matrix $M$ are just
replaced by the column sums $\sum_{\ell=1}^{d_\pi}a_{k\ell}$). So, there is no gain
in using the operator $A'$ instead of the simpler operator $A$. (iii) The commutation
relation in Lemma \ref{lemma_conjugate}(c) is a matrix version of the commutation
relation put into evidence in \cite[Sec.~2]{Tie13}.
\end{Remark}

\begin{proof}[Proof of Lemma \ref{lemma_conjugate}]
(a)  The image of operator $A$ under the isomorphism \eqref{eq_iso} is the operator
$\bigoplus_{k=1}^{d_\pi}a_kH$. So, all the claims follow from standard results on
direct sums of self-adjoint operators (see for instance \cite[p.~268]{RS78}).

(b) One has
$$
\big\|\e^{-ita_kH}M_{k\ell}\e^{ita_\ell H}-M_{k\ell}\big\|_{\B(\ltwo(X,\mu_X))}
=\big\|\e^{-ita_kH}M_{k\ell}\e^{ita_kH}-M_{k\ell}
+M_{k\ell}\big(1-\e^{it(a_k-a_\ell)H}\big)\big\|_{\B(\ltwo(X,\mu_X))}.
$$
Therefore, it is sufficient to show that
$M_{k\ell}\big(1-\e^{it(a_k-a_\ell)H}\big)=0$, which is equivalent to the condition
$$
\big\langle\varphi,
M_{k\ell}\big(1-\e^{it(a_k-a_\ell)H}\big)\varphi\big\rangle_{\ltwo(X,\mu_X)}=0
\quad\hbox{for all }\varphi\in C^\infty(X),
$$
due to the density of $C^\infty(X)$ in ${\ltwo(X,\mu_X)}$. So, take
$\varphi\in C^\infty(X)$, set
$$
F_{k\ell}(t):=\big\langle\varphi,
M_{k\ell}\big(1-\e^{it(a_k-a_\ell)H}\big)\varphi\big\rangle_{\ltwo(X,\mu_X)},
$$
and note that $M_{k\ell}\;\!(a_k-a_\ell)=0$ due to the assumption
$(a_k-a_\ell)(\pi_{\ell k}\circ\phi)\equiv0$ for all $k,\ell\in\{1,\ldots,d_\pi\}$.
Then, one has
$$
\frac\d{\d t}\;\!F_{k\ell}(t)
:=-i\;\!\big\langle\varphi,M_{k\ell}\;\!(a_k-a_\ell)\e^{it(a_k-a_\ell)H}
H\varphi\big\rangle_{\ltwo(X,\mu_X)}
=0.
$$
It follows that $F_{k\ell}(t)=F_{k\ell}(0)=0$ for all $t\in\R$, which proves the
claim.

(c) Using first that $\L_Y$ and $V_1$ commute and then that
$(a_k-a_\ell)(\pi_{\ell k}\circ\phi)\equiv0$ for all $k,\ell\in\{1,\ldots,d_\pi\}$,
one gets for $\varphi_k\in C^\infty(X)$ that
\begin{align*}
&\big(A\;\!U_{\pi,j}-U_{\pi,j}A\big)\sum_{k=1}^{d_\pi}\varphi_k\otimes\pi_{jk}\\
&=-i\sum_{k,\ell=1}^{d_\pi}a_\ell\;\!\big(\L_Y(V_1\varphi_k)\big)
(\pi_{\ell k}\circ\phi)\otimes\pi_{j\ell}
-i\sum_{k,\ell=1}^{d_\pi}a_\ell\;\!(V_1\varphi_k)
\big(\L_Y(\pi_{\ell k}\circ\phi)\big)\otimes\pi_{j\ell}\\
&\qquad+i\sum_{k,\ell=1}^{d_\pi}a_k\big(V_1\big(\L_Y\varphi_k)\big)
(\pi_{\ell k}\circ\phi)\otimes\pi_{j\ell}\\
&=i\sum_{k,\ell=1}^{d_\pi}(a_k-a_\ell)\big(V_1\big(\L_Y\varphi_k)\big)
(\pi_{\ell k}\circ\phi)\otimes\pi_{j\ell}
-i\sum_{k,\ell=1}^{d_\pi}a_\ell\;\!(V_1\varphi_k)
\big(\L_Y(\pi_{\ell k}\circ\phi)\big)\otimes\pi_{j\ell}\\
&=-i\sum_{k,\ell=1}^{d_\pi}a_\ell\;\!(V_1\varphi_k)
\big(\L_Y(\pi_{\ell k}\circ\phi)\big)\otimes\pi_{j\ell}.
\end{align*}
This, together with the fact that $\L_Y(\pi_{\ell k}\circ\phi)\in\linf(X)$ for all
$k,\ell\in\{1,\ldots,d_\pi\}$ and the density of the vectors
$\sum_{k=1}^{d_\pi}\varphi_k\otimes\pi_{jk}$ in $\H^{(\pi)}_j$, implies that
$U_{\pi,j}\in C^1(A)$ with
$$
[A,U_{\pi,j}]\sum_{k=1}^{d_\pi}\varphi_k\otimes\pi_{jk}
=-i\sum_{k,\ell=1}^{d_\pi}a_\ell\;\!(V_1\varphi_k)
\big(\L_Y(\pi_{\ell k}\circ\phi)\big)\otimes\pi_{j\ell},
\quad\varphi_k\in\ltwo(X,\mu_X).
$$
Then, one obtains
\begin{align*}
[A,U_{\pi,j}](U_{\pi,j})^*\sum_{k=1}^{d_\pi}\varphi_k\otimes\pi_{jk}
&=-i\sum_{k,\ell,m=1}^{d_\pi}a_m\big\{V_1\big((V_{-1}\varphi_k)
(\pi_{\ell k}\circ\phi^{(-1)})\big)\big\}\big(\L_Y(\pi_{m\ell}\circ\phi)\big)
\otimes\pi_{jm}\\
&=-i\sum_{k,\ell,m=1}^{d_\pi}a_m\;\!\varphi_k
(\pi^*\circ\phi)_{\ell k}\big(\L_Y(\pi\circ\phi)\big)_{m\ell}
\otimes\pi_{jm}\\
&=M\sum_{k=1}^{d_\pi}\varphi_k\otimes\pi_{jk},
\end{align*}
which shows the equality $[A,U_{\pi,j}]=MU_{\pi,j}$.

To prove that $U_{\pi,j}\in C^{1+0}(A)$, one has to check that
$$
\int_0^1\frac{\d t}t\;\!
\big\|\e^{-itA}[A,U_{\pi,j}]\e^{itA}-[A,U_{\pi,j}]\big\|_{\B(\H^{(\pi)}_j)}<\infty.
$$
But since $[A,U_{\pi,j}]=MU_{\pi,j}$ with $U_{\pi,j}\in C^1(A)$, it is sufficient to
show that
$$
\int_0^1\frac{\d t}t\;\!\big\|\e^{-itA}M\e^{itA}-M\big\|_{\B(\H^{(\pi)}_j)}
<\infty.
$$
Now, Formula \eqref{group_formula} implies that
$$
\big(\e^{-itA}M\e^{itA}-M\big)\sum_{k=1}^{d_\pi}\varphi_k\otimes\pi_{jk}
=\sum_{k,\ell=1}^{d_\pi}\big(\e^{-ita_kH}M_{k\ell}\e^{ita_\ell H}-M_{k\ell}\big)
\varphi_\ell\otimes\pi_{jk},\quad\varphi_k\in\ltwo(X,\mu_X).
$$
It follows that
\begin{align*}
&\int_0^1\frac{\d t}t\;\!\big\|\e^{-itA}M\e^{itA}-M\big\|_{\B(\H^{(\pi)}_j)}\\
&\le\sum_{k,\ell=1}^{d_\pi}\int_0^1\frac{\d t}t\;\!
\big\|\e^{-ita_kH}M_{k\ell}\e^{ita_\ell H}-M_{k\ell}\big\|_{\B(\ltwo(X,\mu_X))}\\
&=\sum_{k,\ell=1}^{d_\pi}\int_0^1\frac{\d t}t\;\!
\big\|\e^{-ita_kH}M_{k\ell}\e^{ita_kH}-M_{k\ell}\big\|_{\B(\ltwo(X,\mu_X))}\\
&=\sum_{k,\ell=1}^{d_\pi}\int_0^{a_k}\frac{\d s}s\;\!
\big\|V_s\;\!M_{k\ell}\;\!V_{-s}-M_{k\ell}\big\|_{\B(\ltwo(X,\mu_X))}\\
&\le{\rm Const.}\sum_{k,\ell=1}^{d_\pi}\int_0^{a_k}\frac{\d s}s\;\!
\left\|\big\{\L_Y(\pi\circ\phi)\cdot(\pi^*\circ\phi)\big\}_{k\ell}\circ F_s
-\big\{\L_Y(\pi\circ\phi)\cdot(\pi^*\circ\phi)\big\}_{k\ell}\right\|_{\linf(X)}\\
&<\infty,
\end{align*}
due to point (b) and the Dini-type condition satisfied by the functions
$\L_Y(\pi_{k\ell}\circ\phi)$.
\end{proof}

In the following theorem, we present a first set of conditions implying a strict
Mourre estimate for $U_{\pi,j}$ on all of $\S^1$ and thus the absolute continuity of
the spectrum of $U_{\pi,j}$. For each $x\in X$, we write $\lambda_k\big(M(x)\big)$,
$k\in\{1,\ldots,d_\pi\}$, for the eigenvalues of the hermitian matrix $M(x)$, and we
use the notation
\begin{equation}\label{lambda_star}
\lambda_*:=\inf_{k\in\{1,\ldots,d_\pi\},\,x\in X}\,\lambda_k\big(M(x)\big).
\end{equation}

\begin{Theorem}[First Mourre estimate for $U_{\pi,j}$]\label{first_theorem}
Let $\phi$ satisfy Assumption \ref{ass_phi}, suppose that
$(a_k-a_\ell)(\pi_{\ell k}\circ\phi)\equiv0$ for all $k,\ell\in\{1,\ldots,d_\pi\}$,
and assume that $\lambda_*>0$. Then, $U_{\pi,j}$ satisfies the strict Mourre estimate
$$
(U_{\pi,j})^*[A,U_{\pi,j}]\ge\lambda_*,
$$
and $U_{\pi,j}$ has purely absolutely continuous spectrum.
\end{Theorem}

\begin{proof}
Since $M$ is hermitian matrix-valued, there exists a function $U:X\to\U(d_\phi)$ such
that $M=U^*DU$ with $D$ diagonal\;\!:
$$
D(x):=
\begin{pmatrix}
\lambda_1\big(M(x)\big) & & 0\\
& \ddots &\\
0 & & \lambda_{d_\pi}\big(M(x)\big)\\
\end{pmatrix},
\quad x\in X.
$$
Furthermore, one has the orthogonality relation
$
\big\langle\pi_{jm},\pi_{jk}\big\rangle_{\ltwo(G,\mu_G)}=\delta_{mk}(d_\pi)^{-1}
$.
Therefore, one obtains for $\varphi_k\in\ltwo(X,\mu_X)$ that
\begin{align*}
\left\langle\sum_{m=1}^{d_\pi}\varphi_m\otimes\pi_{jm},
M\sum_{k=1}^{d_\pi}\varphi_k\otimes\pi_{jk}\right\rangle_{\H^{(\pi)}_j}
&=\sum_{k,\ell=1}^{d_\pi}\big\langle\varphi_k,
M_{k\ell}\;\!\varphi_\ell\big\rangle_{\ltwo(X,\mu_X)}\;\!(d_\pi)^{-1}\\
&=\sum_{m=1}^{d_\pi}\left\langle\left(\sum_{k=1}^{d_\pi}U_{mk}\;\!\varphi_k\right),
\lambda_m\big(M(\;\!\cdot\;\!)\big)\left(\sum_{\ell=1}^{d_\pi}U_{m\ell}\;\!
\varphi_\ell\right)\right\rangle_{\ltwo(X,\mu_X)}(d_\pi)^{-1}\\
&\ge\lambda_*\sum_{m=1}^{d_\pi}\left\langle\left(\sum_{k=1}^{d_\pi}U_{mk}
\;\!\varphi_k\right),\left(\sum_{\ell=1}^{d_\pi}U_{m\ell}\;\!\varphi_\ell\right)
\right\rangle_{\ltwo(X,\mu_X)}(d_\pi)^{-1}\\
&=\lambda_*\left\langle\sum_{m=1}^{d_\pi}\varphi_m\otimes\pi_{jm},
\sum_{k=1}^{d_\pi}\varphi_k\otimes\pi_{jk}\right\rangle_{\H^{(\pi)}_j},
\end{align*}
which is equivalent to the inequality $M\ge\lambda_*$. We thus infer from Lemma
\ref{lemma_conjugate}(c) that $U_{\pi,j}\in C^{1+0}(A)$ with
$$
(U_{\pi,j})^*[A,U_{\pi,j}]=(U_{\pi,j})^*MU_{\pi,j}\ge\lambda_*.
$$
It follows from Theorem \ref{thm_unitary} that $U_{\pi,j}$ has purely absolutely
continuous spectrum.
\end{proof}

Sometimes (as when $F_1$ is uniquely ergodic), it is 	advantageous to replace the
positivity condition $\lambda_*>0$ of Theorem \ref{first_theorem} by an averaged
positivity condition more likely to be satisfied. For this, we have to modify the
conjugate operator $A$. Following the approach of \cite[Sec.~4]{FRT12} and
\cite[Sec.~2]{Tie13}, we use the operator $A_N$ obtained by averaging the operator
$A$ along the flow generated by $U_{\pi,j}:$
\begin{equation}\label{second_operator}
A_N\;\!\varphi
:=\frac1N\sum_{n=0}^{N-1}(U_{\pi,j})^nA\;\!(U_{\pi,j})^{-n}\varphi,
\quad N\in\N_{\ge1},~\varphi\in\dom(A_N):=\dom(A),
\end{equation}
(the operator $A_N$ is self-adjoint on $\dom(A_N)=\dom(A)$ because
$(U_{\pi,j})^n\in C^1(A)$ for each $n\in\Z$, see \cite[Sec.~4]{FRT12}). In such a
case, the averages
\begin{equation}\label{matrix_M_N}
M_N:=\frac1N\sum_{n=0}^{N-1}
\big(\pi\circ\phi^{(n)}\big)(M\circ F_n)\big(\pi^*\circ\phi^{(n)}\big)
\end{equation}
of the operator $M$ appear, and we thus use the notation
\begin{equation}\label{lambda_star_N}
\lambda_{*,N}:=\inf_{k\in\{1,\ldots,d_\pi\},\,x\in X}\,\lambda_k\big(M_N(x)\big).
\end{equation}
Note that if the matrix-valued function $\pi\circ\phi$ is diagonal, then
$\big(\pi\circ\phi^{(n)}\big)$, $(M\circ F_n)$ and $\big(\pi^*\circ\phi^{(n)}\big)$
are also diagonal and $M_N$ reduces to the Birkhoff sum
$$
M_N=\frac1N\sum_{n=0}^{N-1}M\circ F_n.
$$

\begin{Theorem}[Second Mourre estimate for $U_{\pi,j}$]\label{second_theorem}
Let $\phi$ satisfy Assumption \ref{ass_phi}, suppose that
$(a_k-a_\ell)(\pi_{\ell k}\circ\phi)\equiv0$ for all $k,\ell\in\{1,\ldots,d_\pi\}$,
and assume that $\lambda_{*,N}>0$ for some $N\in\N_{\ge1}$. Then, $U_{\pi,j}$
satisfies the strict Mourre estimate
$$
(U_{\pi,j})^*[A_N,U_{\pi,j}]\ge\lambda_{*,N},
$$
and $U_{\pi,j}$ has purely absolutely continuous spectrum.
\end{Theorem}

\begin{proof}
We know from Lemma \ref{lemma_conjugate}(c) that $U_{\pi,j}\in C^{1+0}(A)$. So, it
follows from the abstract result \cite[Lemma~4.1]{FRT12} that
$U_{\pi,j}\in C^{1+0}(A_N)$ with
$
[A_N,U_{\pi,j}]
=\frac1N\sum_{n=0}^{N-1}(U_{\pi,j})^n\;\![A,U_{\pi,j}](U_{\pi,j})^{-n}
$.
Using the equality $[A,U_{\pi,j}]=MU_{\pi,j}$, one thus obtains that
$$
[A_N,U_{\pi,j}]
=\left(\frac1N\sum_{n=0}^{N-1}(U_{\pi,j})^nM(U_{\pi,j})^{-n}\right)U_{\pi,j}
$$
with
\begin{align*}
(U_{\pi,j})^nM(U_{\pi,j})^{-n}\sum_{k=1}^{d_\pi}\varphi_k\otimes\pi_{jk}
&=(U_{\pi,j})^n\sum_{k,\ell,m=1}^{d_\pi}M_{\ell m}(\varphi_k\circ F_{-n})
\big(\pi_{mk}\circ\phi^{(-n)}\big)\otimes\pi_{j\ell}\\
&=\sum_{k,\ell,m,p=1}^{d_\pi}(M_{\ell m}\circ F_n)
\big(\pi_{mk}\circ\phi^{(-n)}\circ F_n\big)
\big(\pi_{p\ell}\circ\phi^{(n)}\big)\varphi_k\otimes\pi_{jp}\\
&=\sum_{k,\ell,m,p=1}^{d_\pi}(M_{\ell m}\circ F_n)\big(\pi^*_{mk}\circ\phi^{(n)}\big)
\big(\pi_{p\ell}\circ\phi^{(n)}\big)\varphi_k\otimes\pi_{jp}\\
&=\sum_{k,p=1}^{d_\pi}\big\{\big(\pi\circ\phi^{(n)}\big)
(M\circ F_n)\big(\pi^*\circ\phi^{(n)}\big)\big\}_{pk}
\;\!\varphi_k\otimes\pi_{jp}
\end{align*}
for $\varphi_k\in\ltwo(X,\mu_X)$. Thus, $U_{\pi,j}\in C^{1+0}(A_N)$ with
$[A_N,U_{\pi,j}]=M_N\;\!U_{\pi,j}$. Since $M_N(x)$ is a hermitian matrix for each
$x\in X$, one can then conclude using the same argument as in the proof of Theorem
\ref{first_theorem}.
\end{proof}

\begin{Remark}[Relation with the topological degree]\label{Rem_degree}
The matrix-valued function $M_N$, which came out from a commutator calculation, is
related to the notion of topological degree of the cocycle $\phi$ in the
representation $\pi$. Indeed, the assumption
$(a_k-a_\ell)(\pi_{\ell k}\circ\phi)\equiv0$ for all $k,\ell\in\{1,\ldots,d_\pi\}$
implies that the matrices $\pi\circ\phi^{(n)}$ and
$D_a:={\sf diag}(a_1,\ldots,a_{d_\pi})$ commute. Thus, one has the equalities
\begin{align*}
M_N
&=D_a\;\!\frac1N\sum_{n=0}^{N-1}
\big(\pi\circ\phi^{(n)}\big)\;\!\L_Y(\pi\circ\phi\circ F_n)(\pi^*\circ\phi\circ F_n)
\big(\pi^*\circ\phi^{(n)}\big)\\
&=D_a\;\!\frac1N\sum_{n=0}^{N-1}
(\pi\circ\phi)\cdots(\pi\circ\phi\circ F_{n-1})
\;\!\L_Y(\pi\circ\phi\circ F_n)(\pi^*\circ\phi\circ F_n)
(\pi^*\circ\phi\circ F_{n-1})\cdots(\pi^*\circ\phi)\\
&=D_a\;\!\frac1N\;\!\L_Y\big((\pi\circ\phi)^{(N)}\big)
\big((\pi\circ\phi)^{(N)}\big)^*,
\end{align*}
and $M_N$ is the product of $D_a\frac1N$ times the matrix-valued function
$\L_Y\big((\pi\circ\phi)^{(N)}\big)\big((\pi\circ\phi)^{(N)}\big)^*$, which can be
associated to the winding number of the curve $(\pi\circ\phi)^{(N)}$ in the unitary
group $\U(d_\pi)\equiv{\sf codomain}(\pi)$ (the Lie derivative $\L_Y$ replaces the
complex derivative of the scalar case). It follows that the limit
$\lim_{N\to\infty}M_N$ (if it exists, in some topology to be specified) can be
interpreted as the matrix topological degree of $\pi\circ\phi$, up to multiplication
by the constant matrix $D_a$.

This furnishes an alternative interpretation to the result of Theorem
\ref{second_theorem}\;\!: If $N$ is large enough, then $M_N$ is close to $D_a$ times
the topological degree of $\pi\circ\phi$. So, the condition $\lambda_{*,N}>0$ means
that the topological degree of $\pi\circ\phi$ has nonzero determinant, and Theorem
\ref{second_theorem} tells us that in this case $U_\phi$ has purely absolutely
continuous spectrum in the subspace associated to $\pi$. This is nothing else but a
local version, in each representation $\pi$, of the result (already known in various
cases, see \cite{Anz51,Fra00,Fra00_2,GLL91,ILR93,Wys04}) that the continuous
component of spectrum of skew products is purely absolutely continuous if $\phi$ is
regular enough and has nonzero topological degree. The main novelty here is that $X$
and $G$ are general compact Lie groups.
\end{Remark}

We conclude the section by noting that in the particular case where
$\{F_t\}_{t\in\R}$ is a translation flow on a torus $X=\T^d\simeq\R^d/\Z^d$, $d\ge1$,
with $F_1$ ergodic along one coordinate, the spectrum of $U_{\pi,j}$ is Lebesgue if
it is purely absolutely continuous. Indeed, assume that
$$
F_t(x):=x+ty~\hbox{(mod $\Z^d$)},\quad t\in\R,~x\in\T^d,
$$
for some $y:=(y_1,\ldots,y_d)\in\R^d$ with $y_{k_0}\in\R\setminus\Q$ for some
$k_0\in\{1,\ldots,d\}$. Let $Q_{k_0}\in\B\big(\H^{(\pi)}_j\big)$ be the unitary
operator given by
$$
\left(Q_{k_0}\sum_{k=1}^{d_\pi}\varphi_k\otimes\pi_{jk}\right)(x,g)
:=\e^{2\pi ix_{k_0}}\sum_{k=1}^{d_\pi}\varphi_k(x)\;\!\pi_{jk}(g),
\quad\varphi_k\in\ltwo(X,\mu_X),~(x,g)\in X\times G,
$$
and let $T_0:\T\to\T$ be the ergodic translation given by $T_0(z):=z+y_{k_0}$.
Finally, denote by $\sigma_\psi$ the spectral measure of $U_{\pi,j}$ associated to a
vector $\psi\in\H^{(\pi)}_j$; that is, the Borel measure on $\T$ defined by the
equalities
$$
\big(\F\sigma_\psi\big)(-m)
=\int_\T\e^{2\pi imz}\d\sigma_\psi(z)
=\big\langle(U_{\pi,j})^m\psi,\psi\big\rangle,\quad m\in\Z,
$$
with $\F$ the Fourier transform. Then, we have the identities
$$
(U_{\pi,j})^m\;\!Q_{k_0}=\e^{2\pi imy_{k_0}}Q_{k_0}(U_{\pi,j})^m
$$
and
$$
\int_\T\e^{2\pi imz}\d\sigma_{Q_{k_0}\psi}(z)
=\big\langle(U_{\pi,j})^mQ_{k_0}\psi,Q_{k_0}\psi\big\rangle_{\H^{(\pi)}_j}
=\e^{-2\pi imy_{k_0}}\big\langle(U_{\pi,j})^m\psi,\psi\big\rangle_{\H^{(\pi)}_j}
=\int_\T\e^{2\pi imz}\d(T_0^*\sigma_\psi)(z)
$$
for all $m\in\Z$ and $\psi\in\H^{(\pi)}_j$. Thus,
$\sigma_{Q_{k_0}\psi}=T_0^*\sigma_\psi$, and one has the following result\;\!:

\begin{Lemma}\label{Lemma_Lebesgue}
Assume that $X=\T^d$, $d\ge1$, and let $\{F_t\}_{t\in\R}$ be given by
$$
F_t(x):=x+ty~\hbox{(mod $\Z^d$)},\quad t\in\R,~x\in\T^d,
$$
for some $y:=(y_1,\ldots,y_d)\in\R^d$ with $y_{k_0}\in\R\setminus\Q$ for some
$k_0\in\{1,\ldots,d\}$. Then, the spectrum of $U_{\pi,j}$ is Lebesgue if it is purely
absolutely continuous.
\end{Lemma}

\begin{proof}
The claim follows from the identity $\sigma_{Q_{k_0}\psi}=T_0^*\sigma_\psi$ and the
ergodicity of $T_0$ (see the proof of \cite[Lemma~3]{ILR93} or
\cite[Lemma~3.1]{Fra00_2} for details).
\end{proof}

\section{Examples}\label{Sec_examples}
\setcounter{equation}{0}

\subsection{The abelian case $X=\T^d$ and $G=\T^{d'}$}\label{Sec_abelian}

Suppose that $X=\T^d$ and $G=\T^{d'}$ for some $d,d'\ge1$, set
$\H:=\ltwo\big(\T^d\times\T^{d'},\mu_{\T^d}\otimes\mu_{\T^{d'}}\big)$, and let
$\{F_t\}_{t\in\R}$ be the translation flow on $\T^d$ given by
$$
F_t(x):=x+ty~\hbox{(mod $\Z^d$)},\quad t\in\R,~x\in\T^d,
$$
for some $y:=(y_1,y_2,\ldots,y_d)\in\R^d$. Then, each element
$\chi_q\in\widehat{\T^{d'}}$ is a $1$-dimensional IUR (character) of $\T^{d'}$ given
by $\chi_q(z):=\e^{2\pi iq\cdot z}$ for some $q\in\Z^{d'}$. One has
$$
\H^{(\chi_q)}_1
=\big\{\varphi\otimes\chi_q\mid\varphi\in\ltwo\big(\T^d,\mu_{\T^d}\big)\big\},
$$
the Lie derivative $\L_Y$ is given by $\L_Y=y\cdot\nabla_x$, and $F_1$ is uniquely
ergodic if and only if the numbers $y_1,y_2,\ldots,y_d,1$ are rationally independent.

Given $q\in\Z^{d'}$, we choose the function $\phi\in C(\T^d;\T^{d'})$ as follows\;\!:

\begin{Assumption}\label{ass_phi_abelian}
The function $\phi\in C(\T^d;\T^{d'})$ satisfies $\phi=\xi+\eta$, where
\begin{enumerate}
\item[(i)] $\xi:\T^d\to\T^{d'}$ is a Lie group homomorphism; that is, $\xi$ is given
by $\xi(x):=Bx$ (mod $\Z^{d'}$) for some $d'\times d$ matrix $B$ with integer
entries,
\item[(ii)] $\eta\in C(\T^d;\T^{d'})$ is such that $\L_Y(q\cdot\eta)$ exists and
satisfies the Dini-type condition
\begin{equation}\label{eq_Dini}
\int_0^1\frac{\d t}t\,\big\|\L_Y(q\cdot\eta)\circ F_t
-\L_Y(q\cdot\eta)\big\|_{\linf(\T^d)}
<\infty.
\end{equation}
\end{enumerate}
\end{Assumption}

Then, the function $\phi$ satisfies Assumption \ref{ass_phi}, the skew product
$T_\phi$ is given by
$$
T_\phi(x,g)=\big(x+y,g+\phi(x)\big),\quad(x,g)\in\T^d\times\T^{d'},
$$
and the matrix-valued function $M$ defined in \eqref{formula_M} reduces to the scalar
function
$$
M=-ia_1\;\!\L_Y(\chi_q\circ\phi)\cdot(\overline{\chi_q}\circ\phi)
=2\pi a_1\;\!\big(\L_Y(q\cdot\xi)+\L_Y(q\cdot\eta)\big)
=2\pi a_1\;\!\big(y\cdot(B^{\sf T}q)+\L_Y(q\cdot\eta)\big).
$$
Accordingly, the matrix-valued function $M_N$ defined in \eqref{matrix_M_N}
reduces to the scalar function
$$
M_N
=\frac1N\sum_{n=0}^{N-1}M\circ F_n
=2\pi a_1\left(y\cdot(B^{\sf T}q)+
\frac1N\sum_{n=0}^{N-1}\L_Y(q\cdot\eta)\circ F_n\right).
$$
So, if $B^{\sf T}q\ne0$ and if the numbers $y_1,y_2,\ldots,y_d,1$ are rationally
independent, one has $y\cdot(B^{\sf T}q)\ne0$. Thus, one can set
$a_1:=\big(2\pi y\cdot(B^{\sf T}q)\big)^{-1}$, so that $M_N$ takes the form 
\begin{equation}\label{abelian_M_N}
M_N=1+\big(y\cdot(B^{\sf T}q)\big)^{-1}
\left(\frac1N\sum_{n=0}^{N-1}\L_Y(q\cdot\eta)\circ F_n\right).
\end{equation}

Collecting what precedes, one obtains the following result on the spectrum of the
operators $U_{\chi_q,1}$ and $U_\phi$ associated to the skew product $T_\phi$.

\begin{Theorem}\label{thm_abelian}
Let $\phi$ satisfy Assumption \ref{ass_phi_abelian}, suppose that $B^{\sf T}q\ne0$,
and assume that $y_1,y_2,\ldots,y_d,1$ are rationally independent. Then,
$U_{\chi_q,1}$ has purely Lebesgue spectrum. In particular, if $\phi$ satisfies
Assumption \ref{ass_phi_abelian} for each $q\in\Z^{d'}$, then the restriction of
$U_\phi$ to the subspace
$
\bigoplus_{q\in\Z^{d'}\!,\,B^{\sf T}q\ne0}\H^{(\chi_q)}_1\subset\H
$
has countable Lebesgue spectrum.
\end{Theorem}

\begin{proof}
We know that $\phi$ satisfies Assumption \ref{ass_phi}, and it is obvious that
$(a_1-a_1)(\chi_q\circ\phi)\equiv0$. Furthermore, due to the unique ergodicity of
$F_1$, we infer from \eqref{abelian_M_N} that
\begin{align*}
\lim_{N\to\infty}M_N
&=1+\big(y\cdot(B^{\sf T}q)\big)^{-1}\lim_{N\to\infty}
\left(\frac1N\sum_{n=0}^{N-1}\L_Y(q\cdot\eta)\circ F_n\right)\\
&=1+\big(y\cdot(B^{\sf T}q)\big)^{-1}\int_{\T^d}\d\mu_{\T^d}\,\L_Y(q\cdot\eta)\\
&=1
\end{align*}
uniformly on $\T^d$. Therefore,
$$
\lambda_{*,N}
=\inf_{x\in\T^d}\lambda_1\big(M_N(x)\big)
=\inf_{x\in\T^d}M_N(x)
>0
$$
if $N$ is large enough. So, it follows from Theorem \ref{second_theorem} and Lemma
\ref{Lemma_Lebesgue} that $U_{\chi_q,1}$ has purely Lebesgue spectrum. The claim on
$U_\phi$ follows from what precedes if one takes into account the separability of the
Hilbert space
$\H\equiv\ltwo\big(\T^d\times\T^{d'},\mu_{\T^d}\otimes\mu_{\T^{d'}}\big)$.
\end{proof}

Theorem \ref{thm_abelian} is consistent with Corollary 4.5 of \cite{Tie13}, where the
same spectral result is obtained using a less general framework. We refer to the
discussion after \cite[Cor.~4.5]{Tie13} for a comparison with prior results on the
spectral analysis of skew products of tori.

\subsection{The case $X=\T^d$ and $G=\SU(2)$}\label{Sec_SU(2)}

Suppose that $X=\T^d$ for some $d\ge1$, let $G=\SU(2)$, set
$\H:=\ltwo\big(\T^d\times\SU(2),\mu_{\T^d}\otimes\mu_{\SU(2)}\big)$, let
$\{F_t\}_{t\in\R}$ be the translation flow on $\T^d$ given by
$$
F_t(x):=x+ty~\hbox{(mod $\Z^d$)},\quad t\in\R,~x\in\T^d,
$$
for some $y:=(y_1,y_2,\ldots,y_d)\in\R^d$, and let $\xi:\T^d\to\SU(2)$ be a Lie group
homomorphism. Then, one has $\L_Y=y\cdot\nabla_x$, the function $\xi$ satisfies
Assumption \ref{ass_phi}, and the skew product $T_\xi$ is given by
\begin{equation}\label{skew_SU_1}
T_\xi(x,g)=\big(x+y,g\;\!\xi(x)\big),\quad(x,g)\in\T^d\times\SU(2).
\end{equation}
Since $\T^d$ is abelian, the range of $\xi$ is contained in a maximal torus of
$\SU(2)$. But, all of these are mutually conjugate to the subgroup
$
\left\{\left(\begin{smallmatrix}
z & 0\\
0 & \overline z
\end{smallmatrix}\right)\right\}_{z\in\S^1}
$
(see \cite[Thm.~IV.1.6 \& Prop.~IV.3.1]{BtD85}). So, we can suppose without loss of
generality that
\begin{equation}\label{condition_xi}
\xi(x)=h
\begin{pmatrix}
\e^{2\pi i(b\cdot x)} & 0\\
0 & \e^{-2\pi i(b\cdot x)}
\end{pmatrix}
h^*,
\quad x\in\T^d,
\end{equation}
for some vector $b\in\Z^d$ and some element $h\in\SU(2)$, and thus that
\begin{equation}\label{eq_xi}
\big(\pi\circ\xi\big)(x)
=\pi(h)\;\!\pi
\begin{pmatrix}
\e^{2\pi i(b\cdot x)} & 0\\
0 & \e^{-2\pi i(b\cdot x)}
\end{pmatrix}
\big(\pi(h)\big)^*
\end{equation}
for each $\pi$, finite-dimensional IUR of $\SU(2)$.

The set $\widehat{\SU(2)}$ of all (equivalence classes of) finite-dimensional IUR's
of $\SU(2)$ can be described as follows (see \cite[Chap.~II]{Sug90}). For each
$n\in\N$, let $V_n$ be the $(n+1)$-dimensional vector space of homogeneous
polinomials of degree $n$ in the variables $z_1,z_2\in\C$. Endow $V_n$ with the basis
$$
p_k(z_1,z_2):=z_1^kz_2^{n-k},\quad k\in\{0,\ldots,n\},
$$
and the scalar product
$\langle\;\!\cdot\;\!,\;\!\cdot\;\!\rangle_{V_n}:V_n\times V_n\to\C$ defined by
$$
\left\langle\sum_{k=0}^n\alpha_k\;\!z_1^kz_2^{n-k},
\sum_{\ell=0}^n\beta_\ell\;\!z_1^\ell z_2^{n-\ell}\right\rangle_{V_n}
:=\sum_{k=0}^nk!(n-k)!\;\!\alpha_k\overline{\beta_k},
\quad\alpha_k,\beta_\ell\in\C.
$$
Then, the function $\pi^{(n)}:\SU(2)\to\U(V_n)\simeq\U(n+1)$ given by
$$
\big(\pi^{(n)}(g)p\big)(z_1,z_2):=p\big(g_{11}z_1+g_{21}z_2,g_{12}z_1+g_{22}z_2\big),
\quad g=
\begin{pmatrix}
g_{11} & g_{12}\\
g_{21} & g_{22}
\end{pmatrix}
\in\SU(2),~p\in V_n,
$$
defines a $(n+1)$-dimensional IUR of $\SU(2)$ on $V_n$, and each finite-dimensional
IUR of $\SU(2)$ is unitarily equivalent to an element of the family
$\{\pi^{(n)}\}_{n\in\N}$ \cite[Prop.~II.1.1 \& Thm.~II.4.1]{Sug90}. A calculation
using the binomial theorem shows that the matrix elements $\pi_{jk}^{(n)}$ of
$\pi^{(n)}$ with respect to the basis $\{p_k\}_{k=0}^n$ satisfy
$$
\pi_{jk}^{(n)}(g)
:=\big\langle p_j,\pi^{(n)}(g)p_k\big\rangle_{V_n}
=j\;\!!(n-j)!\sum_{\ell=0}^k
\begin{pmatrix}
k\\
\ell
\end{pmatrix}
\begin{pmatrix}
n-k\\
j-\ell
\end{pmatrix}
g_{11}^\ell g_{12}^{j-\ell}g_{21}^{k-\ell}g_{22}^{n+\ell-k-j},
\quad j,k\in\{0,\ldots,n\},
$$
with $\big(\begin{smallmatrix}\cdot\\\cdot\end{smallmatrix}\big)$ the binomial
coefficients. In the particular case of diagonal elements
$
g=
\left(\begin{smallmatrix}
g_{11} & 0\\
0 & \overline{g_{11}}
\end{smallmatrix}\right)
\in\SU(2)
$,
we thus get that
\begin{equation}\label{diag_elts}
\pi_{jk}^{(n)}
\begin{pmatrix}
g_{11} & 0\\
0 & \overline{g_{11}}
\end{pmatrix}
=j\;\!!(n-j)!\;\!g_{11}^{2j-n}\delta_{jk},
\quad\delta_{jk}:=
\begin{cases}
1 & \hbox{if}~~j=k\\
0 & \hbox{if}~~j\ne k.
\end{cases}
\end{equation}
Therefore, by replacing $\pi^{(n)}(\;\!\cdot\;\!)$ by the unitarily equivalent
representation ${\big(\pi^{(n)}(h)\big)^*\pi^{(n)}(\;\!\cdot\;\!)\pi^{(n)}(h)}$, we
infer from \eqref{eq_xi} that
$$
\big(\pi_{jk}^{(n)}\circ\xi\big)(x)
=j\;\!!(n-j)!\e^{2\pi i(2j-n)(b\cdot x)}\delta_{jk}.
$$
Then, putting this expression in the formula \eqref{formula_M} for the matrix-valued
function $M$, one obtains that
\begin{align}
M
&=-i
\begin{pmatrix}
a_0 && 0\\
& \ddots &\\
0 && a_n
\end{pmatrix}
\big(\L_Y(\pi^{(n)}\circ\xi)\big)
\cdot\big((\pi^{(n)})^*\circ\xi\big)\nonumber\\
&=-i
\begin{pmatrix}
a_0 && 0\\
& \ddots &\\
0 && a_n
\end{pmatrix}
\big(\pi^{(n)}\circ\xi\big)
\cdot\frac\d{\d t}\big(\pi^{(n)}\circ\xi\big)(ty)\Big|_{t=0}
\cdot\big((\pi^{(n)})^*\circ\xi\big)\nonumber\\
&=2\pi(y\cdot b)
\begin{pmatrix}
a_0\;\!0!\;\!(n-0)!\;\!(2\cdot0-n) && 0\\
& \ddots &\\
0 && a_n\;\!n!\;\!(n-n)!\;\!(2\cdot n-n)
\end{pmatrix}.\label{calc_M}
\end{align}
In the case $y\cdot b\ne0$, we can set
$a_j:=(2j-n)\big(2\pi(y\cdot b)\;\!j\;\!!(n-j)!\big)^{-1}$, and thus obtain that
$$
M_{jk}=(2j-n)^2\;\!\delta_{jk},
$$
which (in view of Equation \eqref{lambda_star}) implies that
$$
\lambda_*
=\inf_{k\in\{0,\ldots,n\}}(2k-n)^2
=
\begin{cases}
0 & \hbox{if}~~n\in2\N\\
1 & \hbox{if}~~n\in2\N+1.
\end{cases}
$$

Collecting what precedes, one ends up with the following result on the spectrum of
the operators $U_{\pi^{(n)}\!,j}$ and $U_\xi$ associated to the skew product $T_\xi$.

\begin{Lemma}\label{first_lemma_SU(2)}
Let $\xi$ satisfy \eqref{condition_xi} with $y\cdot b\ne0$, and take $n\in2\N+1$ and
$j\in\{0,\ldots,n\}$. Then, $U_{\pi^{(n)}\!,j}$ has purely absolutely continuous
spectrum. In particular, the restriction of $U_\xi$ to the subspace
$
\bigoplus_{n\in2\N+1}\bigoplus_{j=0}^n\H^{(\pi^{(n)})}_j\subset\H
$
has purely absolutely continuous spectrum.
\end{Lemma}

\begin{proof}
We know that $\xi$ satisfies Assumption \ref{ass_phi}, that
$(a_k-a_\ell)\big(\pi_{\ell k}^{(n)}\circ\xi\big)\equiv0$ for all
$k,\ell\in\{0,\ldots,n\}$, and that $\lambda_*=1$. So, the claim is a direct
consequence of Theorem \ref{first_theorem}.
\end{proof}

As in Section \ref{Sec_abelian}, we can treat more general cocycles $\phi$ (namely,
perturbations of group homomorphisms) if $F_1$ is uniquely ergodic. So, from now on,
we assume that $y_1,y_2,\ldots,y_d,1$ are rationally independent (so that $F_1$ is
uniquely ergodic) and we suppose that $\phi:\T^d\to\SU(2)$ is a perturbation of $\xi$
in the sense that
\begin{equation}\label{condition_phi_SU(2)}
\phi(x):=h
\begin{pmatrix}
\e^{2\pi i(b\cdot x+\eta(x))} & 0\\
0 & \e^{-2\pi i(b\cdot x+\eta(x))}
\end{pmatrix}
h^*,
\quad x\in\T^d,
\end{equation}
with $b\in\Z^d\setminus\{0\}$ and with $\eta\in C(\T^d;\R)$ satisfying the
following\;\!:

\begin{Assumption}\label{assumption_eta_SU(2)}
$\eta\in C(\T^d;\R)$ is such that $\L_Y\eta$ exists and satisfies the Dini-type
condition
$$
\int_0^1\frac{\d t}t\,\big\|\L_Y\eta\circ F_t-\L_Y\eta\big\|_{\linf(\T^d)}
<\infty.
$$
\end{Assumption}

So, we have
$$
\big(\pi^{(n)}\circ\phi\big)(x)
=\pi^{(n)}(h)\;\!\pi^{(n)}
\begin{pmatrix}
\e^{2\pi i(b\cdot x+\eta(x))} & 0\\
0 & \e^{-2\pi i(b\cdot x+\eta(x))}
\end{pmatrix}
\big(\pi^{(n)}(h)\big)^*,
$$
and, replacing $\pi^{(n)}(\;\!\cdot\;\!)$ by the unitarily equivalent representation
$\big(\pi^{(n)}(h)\big)^*\pi^{(n)}(\;\!\cdot\;\!)\;\!\pi^{(n)}(h)$, we infer from
\eqref{diag_elts} that
$$
\big(\pi_{jk}^{(n)}\circ\phi\big)(x)
=j\;\!!(n-j)!\e^{2\pi i(2j-n)(b\cdot x+\eta(x))}\delta_{jk}.
$$
Therefore, a calculation similar to that of \eqref{calc_M} gives
$$
M_{jk}=2\pi a_j\big((y\cdot b)+\L_Y\eta\big)\;\!j!\;\!(n-j)!\;\!(2j-n)\delta_{jk}.
$$
But now we know that $y\cdot b\ne0$, since $y_1,y_2,\ldots,y_d,1$ are rationally
independent. So, we can set $a_j:=(2j-n)\big(2\pi(y\cdot b)\;\!j!(n-j)!\big)^{-1}$,
and thus obtain that
$$
M_{jk}=\big(1+(y\cdot b)^{-1}\L_Y\eta\big)(2j-n)^2\delta_{jk}.
$$
Accordingly, the matrix-valued function $M_N$ given in \eqref{matrix_M_N} reduces to
\begin{equation}\label{M_N_SU(2)}
M_N
=\frac1N\sum_{m=0}^{N-1}M\circ F_m
=\left(1+(y\cdot b)^{-1}\frac1N\sum_{m=0}^{N-1}\L_Y\eta\circ F_m\right)
\begin{pmatrix}
(2\cdot0-n)^2 && 0\\
& \ddots &\\
0 && (2\cdot n-n)^2
\end{pmatrix}.
\end{equation}

The following theorem on the spectrum of the operators $U_{\pi^{(n)}\!,j}$ and
$U_\phi$ associated to the skew product $T_\phi$ complements Lemma
\ref{first_lemma_SU(2)}.

\begin{Theorem}\label{thm_SU(2)}
Let $\phi$ satisfy \eqref{condition_phi_SU(2)} with $b\in\Z^d\setminus\{0\}$ and
Assumption \ref{assumption_eta_SU(2)}. Suppose that $y_1,y_2,\ldots,y_d,1$ are
rationally independent, and take $n\in2\N+1$ and $j\in\{0,\ldots,n\}$. Then,
$U_{\pi^{(n)}\!,j}$ has purely Lebesgue spectrum. In particular, the restriction of
$U_\phi$ to the subspace
$
\bigoplus_{n\in2\N+1}\bigoplus_{j=0}^n\H^{(\pi^{(n)})}_j\subset\H
$
has countable Lebesgue spectrum.
\end{Theorem}

\begin{proof}
We know that $\phi$ satisfies Assumption \ref{ass_phi} and that
$(a_k-a_\ell)\big(\pi_{\ell k}^{(n)}\circ\phi\big)\equiv0$ for all
$k,\ell\in\{0,\ldots,n\}$. Furthermore, due to the unique ergodicity of $F_1$, we
deduce from \eqref{M_N_SU(2)} that
\begin{align*}
\lim_{N\to\infty}M_N
&=\left(1+(y\cdot b)^{-1}\lim_{N\to\infty}\frac1N\sum_{m=0}^{N-1}\L_Y\eta
\circ F_m\right)
\begin{pmatrix}
(2\cdot0-n)^2 && 0\\
& \ddots &\\
0 && (2\cdot n-n)^2
\end{pmatrix}\\
&=\left(1+(y\cdot b)^{-1}\int_{\T^d}\d\mu_{\T^d}\,\L_Y\eta\right)
\begin{pmatrix}
(2\cdot0-n)^2 && 0\\
& \ddots &\\
0 && (2\cdot n-n)^2
\end{pmatrix}\\
&=\begin{pmatrix}
(2\cdot0-n)^2 && 0\\
& \ddots &\\
0 && (2\cdot n-n)^2
\end{pmatrix}
\end{align*}
uniformly on $\T^d$. Therefore, since $n\in2\N+1$, one has that
$$
\lim_{N\to\infty}\lambda_{*,N}
=\inf_{k\in\{0,\ldots,n\},\,x\in\T^d}\,\lambda_k\left(\lim_{N\to\infty}M_N(x)\right)
=\inf_{k\in\{0,\ldots,n\}}\,\lambda_k
\begin{pmatrix}
(2\cdot0-n)^2 && 0\\
& \ddots &\\
0 && (2\cdot n-n)^2
\end{pmatrix}
=1,
$$
and thus $\lambda_{*,N}>0$ if $N$ is large enough. So, it follows from Theorem
\ref{second_theorem} and Lemma \ref{Lemma_Lebesgue} that $U_{\pi^{(n)}\!,j}$ has
purely Lebesgue spectrum. The claim on $U_\phi$ follows from what precedes if one
takes into account the separability of the Hilbert space
$\H\equiv\ltwo\big(\T^d\times\SU(2),\mu_{\T^d}\otimes\mu_{\SU(2)}\big)$.
\end{proof}

Theorem \ref{thm_SU(2)} should be compared with prior results for skew products on
$\T^d\times\SU(2)$ obtained by K. Fr{\polhk{a}}czek (but see also
\cite{Hou11,Kri01}). When $F_1$ is an ergodic translation on $\T^d$ ($d=1,2$), K.
Fr{\polhk{a}}czek exhibits in \cite[Thm.~6.1 \& Thm.~8.2]{Fra00_2} conditions on the
cocycle $\phi$ guaranteeing that the restriction of $U_\phi$ to the subspace
$\bigoplus_{n\in2\N+1}\bigoplus_{j=0}^n\H^{(\pi^{(n)})}_j$ has countable Lebesgue
spectrum. In dimension $d=1$, these conditions are verified if 
$\phi\in C^2\big(\T;\SU(2)\big)$, if $\phi$ has nonzero topological degree, and if
$\phi$ is cohomologous to a diagonal cocyle with a transfer function
$\zeta:\T\to\SU(2)$ of bounded variation and with
$\zeta'\zeta^{-1}\in\ltwo\big(\T;\frak{su}(2)\big)$ (see \cite[Cor.~6.5]{Fra00_2}).
This is similar (but not completely equivalent) to the conditions satisfied by $\phi$
when $d=1$ in Theorem \ref{thm_SU(2)} (in Theorem \ref{thm_SU(2)}, $\phi$ has
Dini-continuous derivative along the flow $\{F_t\}_{t\in\R}$, it has topological
degree $b\ne0$, and it is cohomologous to a diagonal cocyle with a constant transfer
function). K. Fr{\polhk{a}}czek also shows various properties of the topological
degree of cocycles $\phi\in C^2\big(\T;\SU(2)\big)$ such as the fact that it takes
values in $\Z$ or the fact that it is invariant under the relation of measurable
cohomology (see \cite[Thm.~2.7 \& Thm.~2.10]{Fra04}).

\subsection{The case $X=\T^d$ and $G=\U(2)$}\label{Sec_U(2)}

Suppose that $X=\T^d$ for some $d\ge1$, let $G=\U(2)$, set
$\H:=\ltwo\big(\T^d\times\U(2),\mu_{\T^d}\otimes\mu_{\U(2)}\big)$, let
$\{F_t\}_{t\in\R}$ be the translation flow on $\T^d$ given by
$$
F_t(x):=x+ty~\hbox{(mod $\Z^d$)},\quad t\in\R,~x\in\T^d,
$$
for some $y:=(y_1,y_2,\ldots,y_d)\in\R^d$, take $\xi:\T^d\to\U(2)$ a Lie group
homomorphism, and let
$$
T_\xi(x,g)=\big(x+y,g\;\!\xi(x)\big),\quad(x,g)\in\T^d\times\U(2).
$$
Since $\T^d$ is abelian, the range of $\xi$ is contained in a maximal torus of
$\U(2)$ of the form
$
\left\{h
\left(\begin{smallmatrix}
z_1 & 0\\
0 & z_2
\end{smallmatrix}\right)
h^*\right\}_{z_1,z_2\in\S^1}
$
for some $h\in\U(2)$ (see \cite[Thm.~IV.1.6 \& Prop.~IV.3.1]{BtD85}). So, we can
suppose without loss of generality that
\begin{equation}\label{condition_xi_bis}
\xi(x)=h
\begin{pmatrix}
\e^{2\pi i(b_1\cdot x)} & 0\\
0 & \e^{2\pi i(b_2\cdot x)}
\end{pmatrix}
h^*,\quad x\in\T^d,
\end{equation}
for some vectors $b_1,b_2\in\Z^d$, and thus that
\begin{equation}\label{eq_xi_bis}
\big(\pi\circ\xi\big)(x)
=\pi(h)\;\!\pi
\begin{pmatrix}
\e^{2\pi i(b_1\cdot x)} & 0\\
0 & \e^{2\pi i(b_2\cdot x)}
\end{pmatrix}
\big(\pi(h)\big)^*
\end{equation}
for each $\pi$, finite-dimensional IUR of $\U(2)$.

Now, the fact that the map $\S^1\times\SU(2)\ni(z,g)\mapsto zg\in\U(2)$ is an
epimorphism with kernel $\{(1,e_{\SU(2)}),(-1,-e_{\SU(2)})\}$ implies  that the set
$\widehat{\U(2)}$ of all equivalence classes of finite-dimensional IUR's of $\U(2)$
coincides (up to unitary equivalence) with the set of tensors products
$\{\rho_{2m-n}\otimes\pi^{(n)}\}_{m\in\Z,n\in\N}$, with
$\pi^{(n)}$ as in Section \ref{Sec_SU(2)} and $\rho_{2m-n}:\S^1\to\U(1)\equiv\S^1$
given by $\rho_{2m-n}(z):=z^{2m-n}$ (see \cite[Sec.~II.5]{BtD85}). Therefore, if one
uses \eqref{diag_elts}, \eqref{eq_xi_bis} and the factorisation
$$
\begin{pmatrix}
\e^{2\pi i(b_1\cdot x)} & 0\\
0 & \e^{2\pi i(b_2\cdot x)}
\end{pmatrix}
=
\e^{\pi i(b_+\cdot x)}
\begin{pmatrix}
\e^{\pi i(b_-\cdot x)} & 0\\
0 & \e^{-\pi i(b_-\cdot x)}
\end{pmatrix},
\quad b_\pm:=b_1\pm b_2,
$$
and if one replaces $\big(\rho_{2m-n}\otimes\pi^{(n)}\big)(\;\!\cdot\;\!)$ by the
unitarily equivalent representation
$$
\big(\big(\rho_{2m-n}\otimes\pi^{(n)}\big)(h)\big)^*
\big(\rho_{2m-n}\otimes\pi^{(n)}\big)(\;\!\cdot\;\!)
\big(\rho_{2m-n}\otimes\pi^{(n)}\big)(h),
$$
one obtains that
$$
\big(\big(\rho_{2m-n}\otimes\pi^{(n)}\big)_{jk}\circ\xi\big)(x)
=j\;\!!(n-j)!\e^{\pi i((2m-n)(b_+\cdot x)+(2j-n)(b_-\cdot x))}\delta_{jk},
\quad j,k\in\{0,\ldots,n\}.
$$
Then, putting this expression in the formula \eqref{formula_M} for the matrix-valued
function $M$, one obtains that
$$
M_{jk}=\pi a_j\;\!j\;\!!(n-j)!
\big((2m-n)(b_+\cdot y)+(2j-n)(b_-\cdot y)\big)\;\!\delta_{jk}.
$$
Setting 
$
a_j:=\big((2m-n)(b_+\cdot y)+(2j-n)(b_-\cdot y)\big)\big(\pi\;\!j\;\!!(n-j)!\big)^{-1}
$,
one thus obtains that
$$
M_{jk}=\big((2m-n)(b_+\cdot y)+(2j-n)(b_-\cdot y)\big)^2\;\!\delta_{jk}.
$$
As a consequence, we obtain the following result on the spectrum of the operators
$U_{\rho_{2m-n}\otimes\pi^{(n)}\!,j}$ and $U_\xi$ associated to the skew product
$T_\xi$.

\begin{Lemma}\label{first_lemma_U(2)}
Let $\xi$ satisfy \eqref{condition_xi_bis}, set
\begin{equation}\label{def_R}
R:=\left\{(m,n)\in\Z\times\N\mid
\inf_{k\in\{0,\ldots,n\}}\big((2m-n)(b_+\cdot y)+(2k-n)(b_-\cdot y)\big)^2>0\right\},
\end{equation}
and take $(m,n)\in R$ and $j\in\{0,\ldots,n\}$. Then,
$U_{\rho_{2m-n}\otimes\pi^{(n)}\!,j}$ has purely absolutely continuous spectrum. In
particular, the restriction of $U_\xi$ to the subspace
$
\bigoplus_{(m,n)\in R}\bigoplus_{j=0}^n\H^{(\rho_{2m-n}\otimes\pi^{(n)})}_j\subset\H
$
has purely absolutely continuous spectrum.
\end{Lemma}

\begin{proof}
We know that $\xi$ satisfies Assumption \ref{ass_phi}, that
$(a_k-a_\ell)\big(\pi_{\ell k}^{(n)}\circ\xi\big)\equiv0$ for all
$k,\ell\in\{0,\ldots,n\}$, and
that
$$
\lambda_*=\inf_{k\in\{0,\ldots,n\}}\big((2m-n)(b_+\cdot y)+(2k-n)(b_-\cdot y)\big)^2>0.
$$
So, the claim is a direct consequence of Theorem \ref{first_theorem}.
\end{proof}

\begin{Remark}
Some particular cases of Lemma \ref{first_lemma_U(2)} are worth mentioning. First, if
$b_1=b_2$, then $\lambda_*=4(2m-n)^2(b_1\cdot y)^2$. Thus,
$U_{\rho_{2m-n}\otimes\pi^{(n)}\!,j}$ has purely absolutely continuous spectrum if
$2m\ne n$ and $b_1\cdot y\ne0$. Second, if $b_1=-b_2$, then
$$
\lambda_*
=\inf_{k\in\{0,\ldots,n\}}4(2k-n)^2(b_1\cdot y)^2
=
\begin{cases}
0 & \hbox{if}~~n\in2\N\\
4(b_1\cdot y)^2 & \hbox{if}~~n\in2\N+1.
\end{cases}
$$
Thus, $U_{\rho_{2m-n}\otimes\pi^{(n)}\!,j}$ has purely absolutely continuous spectrum
if $n\in2\N+1$ and $b_1\cdot y\ne0$ (we recover the result of Lemma
\ref{first_lemma_SU(2)}, since $\xi$ takes values in $\SU(2)$). Finally, if $b_1=0$
(or if $b_2=0$; this is similar), then
$$
\lambda_*
=\inf_{k\in\{0,\ldots,n\}}4(m-k)^2(b_2\cdot y)^2
=
\begin{cases}
0 & \hbox{if}~~m\in\{0,\ldots,n\}\\
4(b_2\cdot y)^2 & \hbox{if}~~m\in\Z\setminus\{0,\ldots,n\}.
\end{cases}
$$
Thus, $U_{\rho_{2m-n}\otimes\pi^{(n)}\!,j}$ has purely absolutely continuous spectrum
if $m\in\Z\setminus\{0,\ldots,n\}$ and $b_2\cdot y\ne0$.
\end{Remark}

As in the previous sections, we can treat more general cocycles $\phi$ if $F_1$ is
uniquely ergodic. So, from now on we assume that $y_1,y_2,\ldots,y_d,1$ are
rationally independent and we suppose that $\phi:\T^d\to\U(2)$ is a perturbation of
$\xi$ in the sense that
\begin{equation}\label{condition_phi_U(2)}
\phi(x)=h
\begin{pmatrix}
\e^{2\pi i(b_1\cdot x+\eta_1(x))} & 0\\
0 & \e^{2\pi i(b_2\cdot x+\eta_2(x))}
\end{pmatrix}
h^*,\quad x\in\T^d,
\end{equation}
with $\eta_1,\eta_2\in C(\T^d;\R)$ satisfying the following\;\!:

\begin{Assumption}\label{assumption_eta_U(2)}
For $k=1,2$, the function $\eta_k\in C(\T^d;\R)$ is such that $\L_Y\eta_k$ exists and
satisfies the Dini-type condition
$$
\int_0^1\frac{\d t}t\,\big\|\L_Y\eta_k\circ F_t-\L_Y\eta_k\big\|_{\linf(\T^d)}
<\infty.
$$
\end{Assumption}

If we proceed as before, we obtain that
$$
\big((\rho_{2m-n}\otimes\pi^{(n)})_{jk}\circ\xi\big)(x)\\
=j\;\!!(n-j)!\e^{\pi i\{(2m-n)(b_+\cdot x+\eta_+(x))+(2j-n)(b_-\cdot x+\eta_-(x))\}}
\delta_{jk},\quad\eta_\pm:=\eta_1\pm\eta_2.
$$
Then, calculations similar to those of the previous section lead to the following
result on the spectrum of the operators $U_{\rho_{2m-n}\otimes\pi^{(n)}\!,j}$ and
$U_\phi$ associated to the skew product $T_\phi$ (review \eqref{def_R} for the
definition of $R$).

\begin{Theorem}\label{thm_U(2)}
Let $\phi$ satisfy \eqref{condition_phi_U(2)} and Assumption
\ref{assumption_eta_U(2)}, suppose that $y_1,y_2,\ldots,y_d,1$ are rationally
independent, and take $(m,n)\in R$ and $j\in\{0,\ldots,n\}$. Then,
$U_{\rho_{2m-n}\otimes\pi^{(n)}\!,j}$ has purely Lebesgue spectrum. In particular,
the restriction of $U_\phi$ to the subspace
$
\bigoplus_{(m,n)\in R}\bigoplus_{j=0}^n\H^{(\rho_{2m-n}\otimes\pi^{(n)})}_j\subset\H
$
has countable Lebesgue spectrum.
\end{Theorem}

As far as we know, the result of Theorem \ref{thm_U(2)} is new. Besides, it would be
possible to build on and apply the method of Section \ref{Sec_spectrum} to cocycles
$\phi:X\to G$ taking values in other (higher dimensional) Lie groups than $\SU(2)$ or
$\U(2)$. However, doing this leads one to consider more and more involved
(combinations of tensor products of) families of IUR's in order to prove spectral
results similar to Theorems \ref{thm_SU(2)} and \ref{thm_U(2)}. Thus, we curbed our
enthusiasm, hoping that the transition from $G=\SU(2)$ in Section \ref{Sec_SU(2)} to
$G=\U(2)$ in this section already illustrates the type of procedure one has to
follow.


\def\polhk#1{\setbox0=\hbox{#1}{\ooalign{\hidewidth
  \lower1.5ex\hbox{`}\hidewidth\crcr\unhbox0}}}
  \def\polhk#1{\setbox0=\hbox{#1}{\ooalign{\hidewidth
  \lower1.5ex\hbox{`}\hidewidth\crcr\unhbox0}}}
  \def\polhk#1{\setbox0=\hbox{#1}{\ooalign{\hidewidth
  \lower1.5ex\hbox{`}\hidewidth\crcr\unhbox0}}} \def\cprime{$'$}
  \def\cprime{$'$} \def\polhk#1{\setbox0=\hbox{#1}{\ooalign{\hidewidth
  \lower1.5ex\hbox{`}\hidewidth\crcr\unhbox0}}}
  \def\polhk#1{\setbox0=\hbox{#1}{\ooalign{\hidewidth
  \lower1.5ex\hbox{`}\hidewidth\crcr\unhbox0}}} \def\cprime{$'$}
  \def\cprime{$'$}


\end{document}